\documentclass[11pt,a4paper]{article}
\usepackage[latin1]{inputenc}
\usepackage{amsmath}
\usepackage{amsthm}
\usepackage{amsfonts}
\usepackage{amsfonts,amsthm,latexsym,amsmath,amssymb,amscd,epsfig,psfrag,enumerate}
\usepackage{graphics,graphicx, bezier, float, color, hyperref}
\usepackage{amssymb,url}
\usepackage{multienum}
\usepackage[table]{xcolor}
\usepackage{multicol,multirow}
\usepackage{graphicx}
\usepackage{fancyvrb}
\usepackage{parskip}
\usepackage[toc,page]{appendix}
\usepackage[top=2.7cm, bottom=2.7cm, left=1.5cm, right=1.5cm]{geometry}
\usepackage{xcolor}

\usepackage{blkarray}
\newtheorem{theorem}{Theorem}[section]
\newtheorem{lemma}{Lemma}[section]

\usepackage[none]{hyphenat}[section]
\newtheorem{definition}{Definition}[section]

\numberwithin{equation}{section}
\numberwithin{table}{section}
\numberwithin{figure}{section}

\title{Lucas numbers that are palindromic concatenations of
	two distinct repdigits}
\author{Herbert Batte$^{1,*} $}
\date{}

\begin{document}
\maketitle
\abstract{ Let $ \{L_n\}_{n\geq 0} $ be the sequence of Lucas numbers. In this paper, we determine all Lucas numbers that are palindromic concatenations of two distinct repdigits. } 

 

\thanks{$ ^{*} $ Corresponding author}

\section{Introduction}\label{intro}
\subsection{Background}
\label{sec:1.1}
Consider the Lucas number sequence $\{L_n\}_{n\ge 0}$, which starts with $L_0=2$, $L_1=1$, and follows the pattern $L_{n+2}=L_{n+1}+L_{n}$ for all $n \geq 0$. The initial numbers in this sequence are
$$
2,\;1,\;3,\;4,\;7,\;11,\;18,\;29,\;47,\;76,\;123,\;199,\ldots.
$$
A \textit{repdigit} in base 10 is a positive number $N$ made up of a single repeating digit. Specifically, $N$ is written as
\[
N = \underbrace{\overline{d\cdots d}}_{\ell \text{ times}} = d \left( \frac{10^\ell - 1}{9} \right),
\]
with positive integers \(d\) and \( \ell \), where \(0 \leq d \leq 9\) and \( \ell \geq 1\). Our study adds to the extensive research on the Diophantine characteristics of certain sequences defined by recurrence relations. Particularly, we explore how the terms of these sequences can be expressed as sequences within themselves or as combinations thereof. The work of Luca and Banks \cite{banks}, despite its broad scope, yielded some limited results regarding the count of such sequence terms. The case of Fibonacci numbers composed of two repdigits was addressed in \cite{ala}, with the largest identified as \(F_{14} = 377\).

Recent work has also looked into the relationship between linear recurrence numbers and repdigits. For example, all repdigits formed by the addition of two Padovan numbers were identified in \cite{gar}. This was expanded upon by the Ddamulira in \cite{dda}, who also looked at Padovan numbers that are the concatenation of two different repdigits, finding the largest to be \(P_{21} = 200\) in \cite{dda2}.

Further contributions to this field of research have been made by Bedná\v rik and Trojovská \cite{bed}, Boussayoud et al. \cite{bou}, Bravo and Luca \cite{bravo}, and others \cite{ddam, erd, raya, tro, troj, qu}. The specific findings in \cite{qu} were revisited in \cite{er}, which confirmed that the only Lucas numbers that can be formed by combining two repdigits are
$$ 11,~18, ~29, ~47,~76, ~199, ~322.$$
This finding, also reported in \cite{qu}, was derived through different methods. An interesting follow-up to this work \cite{er} would be to identify Lucas numbers that are \textit{palindromic}. In this context, a number is a \textit{palindrome} if it reads the same backwards as forwards. To begin exploring this, we currently examine a more constrained Diophantine equation:
\begin{align}\label{eq1.1l}
	L_n = \overline{\underbrace{d_1 \ldots d_1}_{\ell \text{ times}}\underbrace{d_2 \ldots d_2}_{m \text{ times}}\underbrace{d_1 \ldots d_1}_{\ell \text{ times}}},
\end{align} 
where \( d_1, d_2 \in \{0, 1, 2, \ldots, 9\}, \) with \( d_1 > 0 \) and \(d_1\ne d_2\). Similar work has been done in \cite{chal} proving that $151$ and $616$ are the only Padovan numbers that are palindromic concatenations of two distinct repdigits. 
\newpage
Here, we present the following result.
\subsection{Main Results}\label{sec:1.2l}
\begin{theorem}\label{thm1.1l} 
	There is no Lucas number which is a palindromic concatenation of two distinct repdigits.
\end{theorem}

\section{Methods}
\subsection{Preliminaries}
Here, we start with the well-known Binet formula for the sequence of Lucas numbers. It is given by
\begin{align}\label{eq2.1l}
	L_n = \alpha^n +\beta^n,~~~\text{where}~~\alpha=\dfrac{1+\sqrt{5}}{2}, ~~\beta=\dfrac{1-\sqrt{5}}{2}.
\end{align}
Note that $\beta=-\alpha^{-1}$ and $|\beta|<1$. It was shown in \cite{BRL} that 
\begin{align}\label{eq2.2l}
	\alpha^{n-1} \le L_n \le2\alpha^n, \quad \text{holds for all} \quad n\ge0.
\end{align}
We go back and rewrite relation \eqref{eq1.1l} as
\begin{align}\label{eq2.3l}
	L_n &= \overline{\underbrace{d_1 \ldots d_1}_{\ell \text{ times}}\underbrace{d_2 \ldots d_2}_{m \text{ times}}\underbrace{d_1 \ldots d_1}_{\ell \text{ times}}}\nonumber\\
	&=\overline{\underbrace{d_1 \ldots d_1}_{\ell \text{ times}}\cdot 10^{\ell+m}+\underbrace{d_2 \ldots d_2}_{m \text{ times}}\cdot 10^{\ell}+\underbrace{d_1 \ldots d_1}_{\ell \text{ times}}}\nonumber\\
	&=\dfrac{1}{9}\left(d_1\cdot 10^{2\ell+m}-(d_1-d_2)\cdot 10^{\ell+m} +(d_1-d_2)\cdot 10^{\ell}-d_1 \right),
\end{align} 
where \( d_1, d_2 \in \{0, 1, 2, \ldots, 9\}\), \( d_1 > 0 \) and \(d_1\ne d_2\).

By equation \eqref{eq2.2l}, assume for a moment that $n>1000$, then equation \eqref{eq2.2l} together with \eqref{eq2.3l} imply that
\begin{align*}
	2\alpha^n\ge L_n >10^{2\ell +m-1},
\end{align*}
and taking logarithms both sides yields $(2\ell+m-1)\log 10<n\log \alpha+\log 2$, which simplifies as
\begin{align*}
	(2\ell+m)\log 10<n\log \alpha+3<n,
\end{align*}
so that
\begin{align}\label{eq2.4l}
	2\ell+m<n,
\end{align}
holds for all $n>1000$.

\subsection{Linear forms in logarithms}
We use three times Baker--type lower bounds for nonzero linear forms in three logarithms of algebraic numbers. There are many such bounds mentioned in the literature like that of Baker and W{\"u}stholz from \cite{BW} or Matveev from \cite{matl}. Before we can formulate such inequalities we need the notion of height of an algebraic number recalled below.

\begin{definition}\label{def2.1l}
	Let $ \gamma $ be an algebraic number of degree $ d $ with minimal primitive polynomial over the integers $$ a_{0}x^{d}+a_{1}x^{d-1}+\cdots+a_{d}=a_{0}\prod_{i=1}^{d}(x-\gamma^{(i)}), $$ where the leading coefficient $ a_{0} $ is positive. Then, the logarithmic height of $ \gamma$ is given by $$ h(\gamma):= \dfrac{1}{d}\Big(\log a_{0}+\sum_{i=1}^{d}\log \max\{|\gamma^{(i)}|,1\} \Big). $$
\end{definition}
 In particular, if $ \gamma$ is a rational number represented as $\gamma:=p/q$ with coprime integers $p$ and $ q\ge 1$, then $ h(\gamma ) = \log \max\{|p|, q\} $. 
The following properties of the logarithmic height function $ h(\cdot) $ will be used in the rest of the paper without further reference:
\begin{equation}\nonumber
	\begin{aligned}
		h(\gamma_{1}\pm\gamma_{2}) &\leq h(\gamma_{1})+h(\gamma_{2})+\log 2;\\
		h(\gamma_{1}\gamma_{2}^{\pm 1} ) &\leq h(\gamma_{1})+h(\gamma_{2});\\
		h(\gamma^{s}) &= |s|h(\gamma)  \quad {\text{\rm valid for}}\quad s\in \mathbb{Z}.
	\end{aligned}
\end{equation}

A linear form in logarithms is an expression
\begin{equation}
	\label{eq:Lambdal}
	\Lambda:=b_1\log \gamma_1+\cdots+b_t\log \gamma_t,
\end{equation}
where for us $\gamma_1,\ldots,\gamma_t$ are positive real  algebraic numbers and $b_1,\ldots,b_t$ are nonzero integers. We assume, $\Lambda\ne 0$. We need lower bounds 
for $|\Lambda|$. We write ${\mathbb K}:={\mathbb Q}(\gamma_1,\ldots,\gamma_t)$ and $D$ for the degree of ${\mathbb K}$.
We start with the general form due to Matveev \cite{matl}. 

\begin{theorem}[Matveev, \cite{matl}]
	\label{thm:Matl} 
	Put $\Gamma:=\gamma_1^{b_1}\cdots \gamma_t^{b_t}-1=e^{\Lambda}-1$. Assume $\Gamma\ne 0$. Then 
	$$
	\log |\Gamma|>-1.4\cdot 30^{t+3}\cdot t^{4.5} \cdot D^2 (1+\log D)(1+\log B)A_1\cdots A_t,
	$$
	where $B\ge \max\{|b_1|,\ldots,|b_t|\}$ and $A_i\ge \max\{Dh(\gamma_i),|\log \gamma_i|,0.16\}$ for $i=1,\ldots,t$.
\end{theorem}

\subsection{Reduction methods}
Typically, the estimates from Matveev's theorem are excessively large to be practical in computations. To refine these estimates, we employ a modified approach based on the Baker--Davenport reduction method. Our adaptation follows the method introduced by Dujella and Pethö (\cite{duj}, Lemma 5a). When considering a real number \( r \), we use \( \| r \| \) to represent the smallest distance between \( r \) and any integer, which is formally written as \( \min\{|r - n| : n \in \mathbb{Z}\} \).
\begin{lemma}[Dujella \& Pethö, \cite{duj}]\label{dujl}
	Let \( \tau \neq 0 \), and \( A, B, \mu \) be real numbers with \( A > 0 \) and \( B > 1 \). Let \( M > 1 \) be a positive integer and suppose that \( p/q \) is a convergent of the continued fraction expansion of \( \tau \) with \( q > 6M \). Let
	\[
	\varepsilon := \| \mu q \| - M \| \tau q \|.
	\]
	If \( \varepsilon > 0 \), then there is no solution of the inequality
	\[
	0 < |m\tau - n + \mu| < AB^{-k}
	\]
	in positive integers \( m, n, k \) with
	\[
	\frac{\log(Aq/\varepsilon)}{\log B} \leq k \quad \text{and} \quad m \leq M.
	\]
\end{lemma}

Finally, we present an analytic argument which is Lemma 7 in \cite{guz}. 
\begin{lemma}[Lemma 7 in \cite{guz}]\label{guzl} If $ s \geq 1 $, $T > (4s^2)^s$ and $T > \displaystyle \frac{z}{(\log z)^s}$, then $$z < 2^s T (\log T)^s.$$	
\end{lemma}
SageMath 9.5 is used to perform all the computations in this work.

\section{Proof of Theorem \ref{thm1.1l}}
\subsection{The low range $n\le 1000$}
Using a basic SageMath script, we investigated all possible solutions to the Diophantine equation \eqref{eq1.1l}. The parameters $d_1, d_2$ were taken from the set $\{0, 1, 2, \ldots, 9\}$ with the conditions that $d_1 > 0$ and $d_1 \neq d_2$, and we restricted our search to $1 \leq \ell, m \leq n \leq 1000$. This search did not result in any solutions, see Appendix 1. From this point, we will only consider cases where $n > 1000$.

\subsection{The case $n> 1000$}
In this subsection, we proceed to examine \eqref{eq2.3l} in three different ways. We first prove the following result.
\begin{lemma}\label{lem3.1l}
	Let $\ell$, $m$ and $n>1000$ be solutions to the Diophantine equation \eqref{eq1.1l}, then
	$$\ell<5\cdot 10^{12}\log n.$$
\end{lemma}
\begin{proof}
We go back to \eqref{eq2.3l} and rewrite it using \eqref{eq2.1l} as	
\begin{align*}
	L_n 
	&=\dfrac{1}{9}\left(d_1\cdot 10^{2\ell+m}-(d_1-d_2)\cdot 10^{\ell+m} +(d_1-d_2)\cdot 10^{\ell}-d_1 \right),\\
	9(\alpha^n+\beta^n) 
	&=d_1\cdot 10^{2\ell+m}-(d_1-d_2)\cdot 10^{\ell+m} +(d_1-d_2)\cdot 10^{\ell}-d_1, \\
	9\alpha^n-d_1\cdot 10^{2\ell+m} &=-9\beta^n-(d_1-d_2)\cdot 10^{\ell+m} +(d_1-d_2)\cdot 10^{\ell}-d_1.
\end{align*}	
Therefore, we have that
\begin{align*}
	\left|9\alpha^n-d_1\cdot 10^{2\ell+m}\right| &=\left|-9\beta^n-(d_1-d_2)\cdot 10^{\ell+m} +(d_1-d_2)\cdot 10^{\ell}-d_1\right|\\
	&\le 9\alpha^{-n}+27\cdot 10^{\ell+m},\quad\text{since}\quad \beta=-\alpha^{-1},\\
	&<28\cdot 10^{\ell+m},
\end{align*}	
where in the last inequality we have used the fact that $n>1000$. Now, dividing both sides by $d_1\cdot 10^{2\ell+m}$, we get	
\begin{align}\label{eq3.1l}
	\left|\dfrac{9}{d_1}\cdot\alpha^n\cdot 10^{-2\ell-m}-1\right|
	&<28\cdot 10^{-\ell}.
\end{align}	
Let 
$$
\Gamma=\dfrac{9}{d_1}\cdot\alpha^n\cdot 10^{-2\ell-m}-1=e^{\Lambda}-1.
$$
Notice that $\Lambda\ne 0$, otherwise we would have
\begin{align*}
	\alpha^n=\dfrac{d_1\cdot 10^{2\ell+m}}{9},
\end{align*}
which is impossible since the left--hand side is irrational and the right--hand side is rational. The algebraic number field containing the following $\gamma_i$'s is $\mathbb{K} := \mathbb{Q}(\sqrt{5})$. We have $D = 2$, $t :=3$,
\begin{equation}\nonumber
	\begin{aligned}
		\gamma_{1}&:=9/d_1,\quad\gamma_{2}:=\alpha,\quad\gamma_{3}:=10,\\
		b_{1}&:=1,\quad \quad~~ b_{2}:=n,\quad b_{3}:=-2\ell-m.
	\end{aligned}
\end{equation}
Since $h(\gamma_{1})=h(9/d_1)\le \log 9 <2.2$, $h(\gamma_{2})=h(\alpha)=0.5\log \alpha <0.25$ and $h(\gamma_{3})=h(10)= \log 10 <2.31$, we take $A_1:=4.4$, $A_2:=0.5$ and $A_3:=4.62$. Next, $B \geq \max\{|b_i|:i=1,2,3\}$. By equation \eqref{eq2.4l}, $2\ell+m<n$, so we take $B:=n$. Now, by Theorem \ref{thm:Matl},
\begin{align}\label{eq3.2l}
	\log |\Gamma| &> -1.4\cdot 30^{6} \cdot 3^{4.5}\cdot 2^2 (1+\log 2)(1+\log n)\cdot 4.4\cdot 0.5\cdot 4.62\nonumber\\
	&> -9.9\cdot 10^{12}(1+\log n).
\end{align}
Comparing \eqref{eq3.1l} and \eqref{eq3.2l}, we get
\begin{align*}
	\ell\log 10-\log 28&<9.9\cdot 10^{12}(1+\log n),\\
	\ell&<4.3\cdot 10^{12}(1+\log n)+1.5\\
	&=4.3\cdot 10^{12}\log n\left(\dfrac{1}{\log n}+1+\dfrac{1.5}{4.3\cdot 10^{12}\log n}\right),
\end{align*}
which leads to $\ell<5\cdot 10^{12}\log n$, for all $n>1000$.	
\end{proof}
Next, we prove the following.
\begin{lemma}\label{lem3.2l}
	Let $\ell$, $m$ and $n>1000$ be solutions to the Diophantine equation \eqref{eq1.1l}, then
	$$m<3\cdot 10^{25}(\log n)^2.$$
\end{lemma}
\begin{proof}
Again, we go back to \eqref{eq2.3l} and rewrite it using \eqref{eq2.1l} as	
\begin{align*}
	L_n 
	&=\dfrac{1}{9}\left(d_1\cdot 10^{2\ell+m}-(d_1-d_2)\cdot 10^{\ell+m} +(d_1-d_2)\cdot 10^{\ell}-d_1 \right),\\
	9(\alpha^n+\beta^n) 
	&=d_1\cdot 10^{2\ell+m}-(d_1-d_2)\cdot 10^{\ell+m} +(d_1-d_2)\cdot 10^{\ell}-d_1, \\
	9\alpha^n-d_1\cdot 10^{2\ell+m} +(d_1-d_2)\cdot 10^{\ell+m}&=-9\beta^n +(d_1-d_2)\cdot 10^{\ell}-d_1.
\end{align*}	
Therefore, we have that
\begin{align*}
	\left|9\alpha^n-(d_1\cdot 10^{\ell}-(d_1-d_2))\cdot 10^{\ell+m}\right| &=\left|-9\beta^n +(d_1-d_2)\cdot 10^{\ell}-d_1\right|\\
	&\le 9\alpha^{-n}+18\cdot 10^{\ell},\quad\text{since}\quad \beta=-\alpha^{-1},\\
	&<19\cdot 10^{\ell},
\end{align*}	
where in the last inequality we have used the fact that $n>1000$. 

Now, dividing both sides by $(d_1\cdot 10^{\ell}-(d_1-d_2))\cdot 10^{\ell+m}$, we get	
\begin{align}\label{eq3.3l}
	\left|\dfrac{9}{(d_1\cdot 10^{\ell}-(d_1-d_2))}\cdot\alpha^n\cdot 10^{-\ell-m}-1\right|
	&<\dfrac{19}{(d_1\cdot 10^{\ell}-(d_1-d_2))}\cdot 10^{-m}<19\cdot 10^{-m}.
\end{align}	
Let 
$$
\Gamma_1=\dfrac{9}{(d_1\cdot 10^{\ell}-(d_1-d_2))}\cdot\alpha^n\cdot 10^{-\ell-m}-1=e^{\Lambda_1}-1.
$$
Notice that $\Lambda_1\ne 0$, otherwise we would have
\begin{align*}
	\alpha^n=\dfrac{(d_1\cdot 10^{\ell}-(d_1-d_2))\cdot 10^{\ell+m}}{9},
\end{align*}
which is impossible since $	\alpha^n$ is irrational and the right--hand side is rational. The algebraic number field containing the following $\gamma_i$'s is $\mathbb{K} := \mathbb{Q}(\sqrt{5})$. Again, we have $D = 2$, $t :=3$,
\begin{equation}\nonumber
	\begin{aligned}
		\gamma_{1}&:=\dfrac{9}{(d_1\cdot 10^{\ell}-(d_1-d_2))},\quad\gamma_{2}:=\alpha,\quad\gamma_{3}:=10,\\
		b_{1}&:=1,\quad \quad\quad\quad\quad\quad\quad\quad\quad\quad b_{2}:=n,~~~ b_{3}:=-\ell-m.
	\end{aligned}
\end{equation}
In order to determine what $A_1$ will be, we need to find the find the maximum of the quantities $h(\gamma_{1} )$ and $|\log \gamma_1 |$. We note that
\begin{align*}
	h(\gamma_{1})&=h\left(\dfrac{9}{d_1\cdot 10^{\ell}-(d_1-d_2)}\right) \le h(9)+h(d_1)+\ell h(10)+h(d_1-d_2)+\log 2 \le 4\log 9+\ell \log 10\\
	&<4\log 9+(5\cdot 10^{12}\log n)\log 10<1.2\cdot 10^{13}\log n,
\end{align*}
where we used Lemma \ref{lem3.1l} and the fact that $n>1000$. On the other note,  
\begin{align*}
	|\log \gamma_1 |&=\left|\log\left(\dfrac{9}{d_1\cdot 10^{\ell}-(d_1-d_2)}\right)\right| \le \log 9+\left|\log (d_1\cdot 10^{\ell}-(d_1-d_2))\right|\\
	&\le  \log 9+\log \left(d_1\cdot 10^{\ell}\right)+\left|\log \left(1-\dfrac{(d_1-d_2)}{d_1\cdot 10^{\ell}}\right)\right|\\
	&\leq \log 9 + \log d_1+\ell \log 10   + \left| \frac{|d_1 - d_2|}{d_1 \cdot 10^\ell} + \frac{1}{2} \left( \frac{|d_1 - d_2|}{d_1 \cdot 10^\ell} \right)^2 + \cdots \right| \\
	&\leq 2\log 9+\ell \log 10  + \frac{1}{10^\ell} + \frac{1}{2 \cdot 10^{2\ell}} + \cdots \\
	&< 2\log 9+ (5 \cdot 10^{12} \log n)\log 10 + \frac{1}{10^\ell - 1} < 1.16 \cdot 10^{13}  \log n,	
\end{align*}
where still we used Lemma \ref{lem3.1l} and the fact that $n>1000$. Since $Dh(\gamma_{1})>|\log \gamma_1 |$, we can take $A_1:=2.4 \cdot 10^{13}  \log n$. As before in the proof of Lemma \ref{lem3.1l}, we can still take $A_2:=0.5$, $A_3:=4.62$ and $B:=n$. Now, by Theorem \ref{thm:Matl},
\begin{align}\label{eq3.4l}
	\log |\Gamma| &> -1.4\cdot 30^{6} \cdot 3^{4.5}\cdot 2^2 (1+\log 2)(1+\log n)\cdot 0.5\cdot 4.62\cdot 2.4 \cdot 10^{13}  \log n\nonumber\\
	&> -6.2\cdot 10^{25}(\log n)^2.
\end{align}
Comparing \eqref{eq3.3l} and \eqref{eq3.4l}, we get
\begin{align*}
	m\log 10-\log 19&<6.2\cdot 10^{25}(\log n)^2,\\
	m&<2.7\cdot 10^{25}(\log n)^2+1.3
	=2.7\cdot 10^{25}(\log n)^2\left(1+\dfrac{1.3}{2.7\cdot 10^{25}(\log n)^2}\right),
\end{align*}
which leads to $m<3\cdot 10^{25}(\log n)^2$, for all $n>1000$.	
\end{proof}
Lastly in this subsection, we prove the following.
\begin{lemma}\label{lem3.3l}
	Let $\ell$, $m$ and $n>1000$ be solutions to the Diophantine equation \eqref{eq1.1l}, then
	$$\ell< 5.3\cdot 10^{14},\quad m<3.4\cdot 10^{29}\quad\text{and}\quad n<9\cdot 10^{45}.$$
\end{lemma}

\begin{proof}
	Once more, we revisit equation \eqref{eq2.3l} and rewrite it using \eqref{eq2.1l} as	
	\begin{align*}
		9\alpha^n-d_1\cdot 10^{2\ell+m} +(d_1-d_2)\cdot 10^{\ell+m}-(d_1-d_2)\cdot 10^{\ell}&=-9\beta^n -d_1.
	\end{align*}	
	Therefore, we have that
	\begin{align*}
		\left|9\alpha^n-(d_1\cdot 10^{\ell+m}-(d_1-d_2)\cdot 10^{m}+(d_1-d_2))\cdot 10^{\ell}\right| &=\left|-9\beta^n -d_1\right|\le 9\alpha^{-n}+9<10,
	\end{align*}	
	where we have used the fact that $\beta=-\alpha^{-1}$ and $n>1000$. Now, dividing sides both by $9\alpha^n$, we get	
	\begin{align}\label{eq3.5l}
		\left|\dfrac{(d_1\cdot 10^{\ell+m}-(d_1-d_2)\cdot 10^{m}+(d_1-d_2))}{9}\cdot \alpha^{-n}\cdot 10^{\ell}-1\right|
		&<\dfrac{10}{9\alpha^{n}}.
	\end{align}	
	Let 
	$$
	\Gamma_2=\dfrac{(d_1\cdot 10^{\ell+m}-(d_1-d_2)\cdot 10^{m}+(d_1-d_2))}{9}\cdot \alpha^{-n}\cdot 10^{\ell}-1=e^{\Lambda_2}-1.
	$$
	Notice that $\Lambda_2\ne 0$, otherwise we would have
	\begin{align*}
		\alpha^n=\dfrac{(d_1\cdot 10^{\ell+m}-(d_1-d_2)\cdot 10^{m}+(d_1-d_2))}{9}\cdot 10^{\ell},
	\end{align*}
	which is impossible since $	\alpha^n$ is irrational and the right--hand side is rational. The algebraic number field containing the following $\gamma_i$'s is $\mathbb{K} := \mathbb{Q}(\sqrt{5})$. Again, we have $D = 2$, $t :=3$,
	\begin{equation}\nonumber
		\begin{aligned}
			\gamma_{1}&:=\dfrac{(d_1\cdot 10^{\ell+m}-(d_1-d_2)\cdot 10^{m}+(d_1-d_2))}{9},\quad\gamma_{2}:=\alpha,\quad\gamma_{3}:=10,\\
			b_{1}&:=1,\quad \quad\quad\quad\quad\quad \quad\quad\quad\quad\quad \quad\quad\quad\quad\quad\quad\quad~~ b_{2}:=-n,~~ b_{3}:=-\ell.
		\end{aligned}
	\end{equation}
Again, in order to determine what $A_1$ will be here, we need to find the find the maximum of the quantities $h(\gamma_{1} )$ and $|\log \gamma_1 |$. We note that
	\begin{align*}
		h(\gamma_{1})&=h\left(\dfrac{d_1\cdot 10^{\ell+m}-(d_1-d_2)\cdot 10^{m}+(d_1-d_2)}{9}\right) \\
		&\le h(9)+h(d_1)+(\ell+m) h(10)+h(d_1-d_2)+mh(10)+h(d_1-d_2)+3\log 2 \\
		&\le 7\log 9+(\ell+m) \log 10+m \log 10\\
		&<7\log 9+\left(3.1\cdot 10^{25}(\log n)^2\right)\log 10+\left(3\cdot 10^{25}(\log n)^2\right)\log 10
		<1.41\cdot 10^{26}(\log n)^2,
	\end{align*}
	where we used Lemmas \ref{lem3.1l}, \ref{lem3.2l} and the fact that $n>1000$. 
	
	On the other note,  
	\begin{align*}
		|\log \gamma_1 |&=\left|\log\left(\dfrac{d_1\cdot 10^{\ell+m}-(d_1-d_2)\cdot 10^{m}+(d_1-d_2)}{9}\right)\right| \\
		&\le \log 9+\left|\log (d_1\cdot 10^{\ell+m}-(d_1-d_2)\cdot 10^{m}+(d_1-d_2))\right|\\
		&\le  \log 9+\log \left(d_1\cdot 10^{\ell+m}\right)+\left|\log \left(1-\dfrac{(d_1-d_2)(10^m-1)}{d_1\cdot 10^{\ell+m}}\right)\right|\\
		&\leq \log 9 + \log d_1+(\ell+m) \log 10   + \left| \dfrac{(d_1-d_2)(10^m-1)}{d_1\cdot 10^{\ell+m}} + \frac{1}{2} \left( \dfrac{(d_1-d_2)(10^m-1)}{d_1\cdot 10^{\ell+m}} \right)^2 + \cdots \right| \\
		&\leq 2\log 9+(\ell+m)  \log 10  + \frac{1}{10^\ell} + \frac{1}{2 \cdot 10^{2\ell}} + \cdots \\
		&< 2\log 9+ \left(3.1\cdot 10^{25}(\log n)^2\right)\log 10 + \frac{1}{10^\ell - 1} < 7.2 \cdot 10^{25}  (\log n)^2,	
	\end{align*}
	where still we used Lemmas \ref{lem3.1l}, \ref{lem3.2l} and the fact that $n>1000$. Since $Dh(\gamma_{1})>|\log \gamma_1 |$, we can take $A_1:=2.82\cdot 10^{26}(\log n)^2$. As before in the proof of Lemma \ref{lem3.1l}, we can still take $A_2:=0.5$, $A_3:=4.62$ and $B:=n$. Now, by Theorem \ref{thm:Matl},
	\begin{align}\label{eq3.6l}
		\log |\Gamma| &> -1.4\cdot 30^{6} \cdot 3^{4.5}\cdot 2^2 (1+\log 2)(1+\log n)\cdot 0.5\cdot 4.62\cdot 2.82\cdot 10^{26}(\log n)^2\nonumber\\
		&> -7.24\cdot 10^{38}(\log n)^3.
	\end{align}
	Comparing \eqref{eq3.5l} and \eqref{eq3.6l}, we get
	\begin{align*}
		n\log\alpha-\log (10/9)&<7.24\cdot 10^{38}(\log n)^3,\\
		n&<1.51\cdot 10^{39}(\log n)^3+0.22
		=1.51\cdot 10^{39}(\log n)^3\left(1+\dfrac{0.22}{1.51\cdot 10^{39}(\log n)^3}\right),
	\end{align*}
	which leads to $n<1.52\cdot 10^{39}(\log n)^3$, for all $n>1000$. 	To proceed from here, let $ s =3\geq 1 $, $T =1.52\cdot 10^{39}> (4s^2)^s=46656$ and $z=n$, then Lemma \ref{guzl} implies that $n < 2^3 \cdot 1.52\cdot 10^{39} (\log 1.52\cdot 10^{39})^3$, or $n<9\cdot 10^{45}$.	
	
Moreover, Lemma \ref{lem3.1l} gives $\ell<5\cdot 10^{12}\log n<5\cdot 10^{12}\log (9\cdot 10^{45})<5.3\cdot 10^{14}$ and Lemma \ref{lem3.2l} gives $m<3\cdot 10^{25}(\log n)^2<3\cdot 10^{25}(\log (9\cdot 10^{45}))^2<3.4\cdot 10^{29}$.
\end{proof}
The bounds established in Lemma \ref{lem3.3l} exceed practical computational utility and require reduction. This process is detailed in Subsection \ref{subsecl}.

\subsection{The reduction process}\label{subsecl}
Here, we apply Lemma \ref{dujl} as follows. First, we return to the inequality \eqref{eq3.1l} and put
\[
\Lambda_1 := (2\ell + m) \log 10 - n \log \alpha + \log \left( \frac{d_1}{9} \right).
\]
Inequality \eqref{eq3.1l} can be rewritten as
$|\Gamma_1| = |e^{\Lambda_1} - 1| < 28\cdot10^{-\ell}$. If we assume that \(\ell \geq 2\), then the right--hand side of this inequality is at most \(0.28 < 1/2\). The inequality \(|e^{\Lambda_1} - 1| < x\) for real values of \(x\) and \({\Lambda_1}\) implies that \(|{\Lambda_1}| < 2x\). Thus,
$|\Lambda_1| < 56\cdot10^{-\ell}$. This implies that
\[
\left| (2\ell + m) \log 10 - n \log \alpha + \log \left( \frac{d_1}{9}\right) \right| < 56\cdot10^{-\ell}.
\]
Dividing through the above inequality by \(\log \alpha\) gives
\[
\left| (2\ell + m) \frac{\log 10}{\log \alpha} - n + \left( \frac{\log(d_1/9)}{\log \alpha} \right) \right| < \frac{56}{\log \alpha}\cdot 10^{-\ell}.
\]
So, we apply Lemma \ref{dujl} with the quantities:
\[
\tau := \frac{\log 10}{\log \alpha}, \quad \mu(d_1) := \frac{\log(d_1/9)}{\log \alpha}, \quad 1 \leq d_1 \leq 9, \quad A := \frac{56}{\log \alpha}, \quad B := 10.
\]
Let \(\tau = [a_0; a_1, a_2, \ldots] = [4; 1, 3, 1, 1, 1, 6, 4, 2, 1, 10, 1, 4, 46, 3, 1, 2, 1, 2, \ldots]\) be the continued fraction expansion of \(\tau\). We set \(M := 10^{46}\) which is an upper bound of \(2\ell + m\). With the help of SageMath in Appendix 2, we find that the convergent
\[
\dfrac{p}{q} = \dfrac{p_{98}}{q_{98}}   = \dfrac{1645685064668785741047746957258993430046006088389}{343927838259763182336125476035118084206130771252} ,
\]
is such that \(q = q_{98} > 6M\). Furthermore, it gives \(\varepsilon > 0.4614141430\), and thus,
\[
\ell \leq \frac{\log((56/\log \alpha)q/\varepsilon)}{\log 10} < 52.
\]
Therefore, we have that \(\ell < 52\). The case \(\ell < 2\) also holds because \(\ell < 2 < 52\).

Next, for fixed \( d_1, d_2 \in \{0, 1, 2, \ldots, 9\} \), \( d_1 > 0 \), \(d_1 \neq d_2\) and \( 1 \leq \ell < 52 \), we return to inequality \eqref{eq3.3l} and put
\[
\Lambda_2 := (\ell + m) \log 10 - n \log \alpha + \log \left( \frac{d_1 \cdot 10^\ell - (d_1 - d_2)}{9} \right).
\]
From inequality \eqref{eq3.3l}, we have that
\(
|\Gamma_2| = |e^{\Lambda_2} - 1| < 19\cdot10^{-m}.
\)
Assume that \( m \geq 2 \), then the right--hand side of this inequality is at most \( 0.19 < 1/2 \). Thus,
$|\Lambda_2| < 38\cdot10^{-m}$,
which implies that
\[
\left| (\ell + m) \log 10 - n \log \alpha + \log \left( \frac{d_1 \cdot 10^\ell - (d_1 - d_2)}{9} \right) \right| < 38\cdot10^{-m}.
\]
Dividing through the above inequality by \( \log \alpha \) gives
\[
\left| (\ell + m) \frac{\log 10}{\log \alpha} - n + \frac{\log ((d_1 \cdot 10^\ell - (d_1 - d_2))/9)}{\log \alpha} \right| < \frac{38}{ \log \alpha}\cdot 10^{-m}.
\]
Thus, we apply Lemma \ref{dujl} with the quantities:
\[
\mu(d_1, d_2) := \frac{\log ((d_1 \cdot 10^\ell - (d_1 - d_2))/9)}{\log \alpha}, \quad A := \frac{38}{\log \alpha}, \quad B := 10.
\]
We take the same \( \tau \) as before and its convergent \( p/q = p_{98}/q_{98} \). Since \( \ell + m < 2\ell + m \), we set \( M := 10^{46} \) as an upper bound on \( \ell + m \). With the help of a simple computer program in SageMath (Appendix 3), we get that \( \varepsilon > 0.4906425804 \), and therefore,
\[
m \leq \frac{\log((38/\log \alpha)/\varepsilon)}{\log 10} < 54.
\]
Hence, we have that \( m < 54 \). The case \( m < 2 \) holds as well since \( m < 2 < 54 \).

Lastly, for fixed \( d_1, d_2 \in \{0, 1, 2, \ldots, 9\} \), \( d_1 > 0 \), \(d_1 \neq d_2\), \( 1 \leq \ell <52 \) and \( 1 \leq m < 54 \), we return to inequality \eqref{eq3.5l} and put
\[
\Lambda_3 := \ell \log 10 - n \log \alpha + \log \left( \frac{d_1 \cdot 10^{\ell+m} - (d_1 - d_2) \cdot 10^m + (d_1 - d_2)}{9} \right).
\]
From inequality \eqref{eq3.5l}, we have that $
|\Gamma_3| = |e^{\Lambda_3} - 1| < 10/9\alpha^n$. Since \( n > 1000 \), the right--hand side of this inequality is less than \( 1/2 \). Thus, the above inequality implies that
$
|\Lambda_3| < 20/9\alpha^n,
$
which leads to
\[
\left| \ell \log 10 - n \log \alpha + \log \left( \frac{d_1 \cdot 10^{\ell+m} - (d_1 - d_2) \cdot 10^m + (d_1 - d_2)}{9} \right) \right| < \frac{20}{9\alpha^n}.
\]
Dividing through the above inequality by \( \log \alpha \) gives,
\[
\left| \ell\frac{ \log 10}{\log \alpha} - n + \frac{\log \left( (d_1 \cdot 10^{\ell+m} - (d_1 - d_2) \cdot 10^m + (d_1 - d_2))/9 \right)}{\log \alpha} \right| < \frac{20}{ 9\log \alpha}\cdot \alpha^{-n}.
\]
Again, we apply Lemma \ref{dujl} with the quantities:
\[
\mu(d_1, d_2) := \frac{\log \left( (d_1 \cdot 10^{\ell+m} - (d_1 - d_2) \cdot 10^m + (d_1 - d_2))/9 \right)}{\log \alpha}, \quad A := \frac{20}{9\log \alpha}, \quad B := \alpha.
\]
We can still take the same \( \tau \) and its convergent \( p/q = p_{98}/q_{98} \) as before. Since \( \ell < 2\ell + m \), we choose \( M := 10^{46} \) as an upper bound for \( \ell \). With the help of a simple computer program in SageMath (Appendix 4), we get that \( \varepsilon > 0.4929934686 \), and thus,
\[
n \leq \frac{\log((20/9\log \alpha)/\varepsilon)}{\log \alpha} < 269,
\]
contradicting our working assumption that \( n > 1000 \), hence Theorem \ref{thm1.1l} holds.\qed

\section*{Conclusion}
In this work, we showed that there is no Lucas numbers that is a palindromic concatenation of two distinct repdigits. It remains an open problem to determine such palindromes in generalized $k-$Lucas numbers. 
\section*{Acknowledgments} 
The author thanks the Eastern Africa Universities Mathematics Programme (EAUMP) for funding his doctoral studies.

\section*{Address}
$ ^{1} $ Department of Mathematics, School of Physical Sciences, College of Natural Sciences, Makerere University, Kampala, Uganda

Email: \url{hbatte91@gmail.com}
\pagebreak
\section*{Appendices}
\subsection*{Appendix 1}\label{app1}
\begin{verbatim}
def generate_lucas_numbers(limit):
    lucas = [2, 1]  # Initial Lucas numbers
    while len(lucas) < limit:
        lucas.append(lucas[-1] + lucas[-2])
    return set(lucas)  # Use a set for efficient membership checking

def construct_palindromic_numbers():
    palindromic_numbers = set()
    for d1 in range(1, 10):  # d1 is a nonzero digit
       for d2 in range(10):  # d2 can be any digit, including zero
          if d1 != d2:
             for length in range(1, 100):  # Reasonable lengths for repdigits
                  first_part = str(d1) * length
                  second_part = str(d2) * length
                  palindromic_number = int(first_part + second_part + first_part)
                  palindromic_numbers.add(palindromic_number)
    return palindromic_numbers

# Generate Lucas numbers and palindromic numbers
lucas_numbers = generate_lucas_numbers(1001)
palindromic_numbers = construct_palindromic_numbers()

# Find the intersection of the two sets
palindromic_lucas = lucas_numbers.intersection(palindromic_numbers)

print("Lucas numbers which are palindromic concatenations of two distinct repdigits:")
print(sorted(palindromic_lucas))

\end{verbatim}
\subsection*{Appendix 2}\label{app2}
\begin{verbatim}
from sage.all import *

# Constants
a = golden_ratio.n(digits=1000)  # Using the golden ratio as an approximation for the 
                                 # root of x^2 - x - 1
tau = (log(10) / log(a)).n(digits=1000)
A = (56 / log(a))
B = 10
M = 1 * 10^46

# Continued Fraction and Convergents
cf_tau = continued_fraction(tau)
convergents = cf_tau.convergents()

for d1 in range(1, 10):  # Iterate through d1 from 1 to 9
    mu = (log(d1/9) / log(a)).n(digits=1000)

    DD = []  # Initialize empty list for results for each d1

    # Dujella and Pethö Reduction Method
    for i, convergent in enumerate(convergents):
        p, q = convergent.numerator(), convergent.denominator()
        ep = abs(mu * q - round(mu * q)) - M * abs(tau * q - round(tau * q))

        if q > 6 * M and ep > 0:
            log_expr_a = (log(A * q / ep) / log(B)).n(digits=10)
            DD.append((i, ep.n(digits=10), log_expr_a))
            print(f"d1 = {d1}, p_{i}/q_{i} = {p}/{q}")
            break  # Stop after finding the first suitable convergent for this d1

    # Results for each d1
    if DD:
        print(f"Results for d1 = {d1}:")
        print("First few elements of DD:", DD[:1])
    else:
        print(f"No suitable convergent found for d1 = {d1}.")
print("Continued fraction expansion of tau:", cf_tau[:20])
\end{verbatim}

\subsection*{Appendix 3}\label{app3}
\begin{verbatim}
from sage.all import *

# Constants
a = golden_ratio.n(digits=1000)  # Using the golden ratio
tau = (log(10) / log(a)).n(digits=1000)
A = (38 / log(a))
B = 10
M = 1 * 10^46

# Continued Fraction and Convergents
cf_tau = continued_fraction(tau)
convergents = cf_tau.convergents()

# Variables to store maximum values
max_ep = -Infinity
max_log_expr_a = -Infinity

for d1 in range(1, 10):
    for d2 in range(0, 10):
        if d1 == d2:
           continue  # d1 should not be equal to d2
        for l in range(1, 52):
           mu = (log((d1 * 10^l - (d1 - d2)) / 9) / log(a)).n(digits=1000)

           # Dujella and Pethö Reduction Method
           for i, convergent in enumerate(convergents):
               p, q = convergent.numerator(), convergent.denominator()
               ep = abs(mu * q - round(mu * q)) - M * abs(tau * q - round(tau * q))

               if q > 6 * M and ep > 0:
                  log_expr_a = (log(A * q / ep) / log(B)).n(digits=10)
                  if ep > max_ep:
                     max_ep = ep.n(digits=10)
                  if log_expr_a > max_log_expr_a:
                     max_log_expr_a = log_expr_a
                  break  # Stop after finding the first suitable convergent

# Print maximum values
print("Maximum ep across all combinations:", max_ep)
print("Maximum log_expr_a across all combinations:", max_log_expr_a)
\end{verbatim}

\subsection*{Appendix 4}\label{app4}
\begin{verbatim}
from sage.all import *

# Constants
a = golden_ratio.n(digits=1000)  # Using the golden ratio
tau = (log(10) / log(a)).n(digits=1000)
A = (20 / (9 * log(a)))
B = a
M = 1 * 10^46

# Continued Fraction and Convergents
cf_tau = continued_fraction(tau)
convergents = cf_tau.convergents()

# Variables to store maximum values
max_ep = -Infinity
max_log_expr_a = -Infinity

# Iterate through combinations of d1, d2, l, and m
for d1 in range(1, 10):
    for d2 in range(0, 10):
      if d1 == d2:
        continue  # d1 should not be equal to d2
      for l in range(1, 52):
         for m in range(1, 54):
            mu = (log((d1 * 10^(l+m) - (d1 - d2) * 10^m + 
            (d1 - d2)) / 9) / log(a)).n(digits=1000)

            # Dujella and Pethö Reduction Method
            for convergent in convergents:
                p, q = convergent.numerator(), convergent.denominator()
                ep = abs(mu * q - round(mu * q)) - M * abs(tau * q - round(tau * q))

                if q > 6 * M and ep > 0:
                   log_expr_a = (log(A * q / ep) / log(B)).n(digits=10)
                   if ep > max_ep:
                      max_ep = ep.n(digits=10)
                   if log_expr_a > max_log_expr_a:
                      max_log_expr_a = log_expr_a
                   break  # Stop after finding the first suitable convergent

# Print maximum values
print("Maximum ep across all combinations:", max_ep)
print("Maximum log_expr_a across all combinations:", max_log_expr_a)	
\end{verbatim}	
\end{document}